\documentclass[12pt, openany]{amsart}
\usepackage[utf8]{inputenc}
\usepackage[T1]{fontenc}
\usepackage[pdftex]{graphicx}
\usepackage{amsmath,amssymb,amsthm}
\usepackage{mathrsfs}
\usepackage{hyperref}
\usepackage{pgf,tikz,pgfplots}
\usetikzlibrary{arrows}
\usepackage{graphicx}
\usepackage{array}
\usepackage{color, marginnote}

\newtheorem{theorem}{Theorem}
\theoremstyle{definition}
\newtheorem{definition}[theorem]{Definition}
\theoremstyle{remark}

\theoremstyle{remark}
\newtheorem{remark}[theorem]{Remark}
\theoremstyle{theorem}
\newtheorem{lemma}[theorem]{Lemma}
\theoremstyle{theorem}
\newtheorem{conjecture}[theorem]{Conjecture}
\theoremstyle{theorem}
\newtheorem{proposition}[theorem]{Proposition}
\theoremstyle{theorem}
\newtheorem{corollary}[theorem]{Corollary}

\newcommand{\restrict}[2]{\left.#1\right|_{#2}}

\newcommand{\graph}[1]{\text{graph}\,#1}
\newcommand{\class}{\mathcal{C}}

\newcommand{\NN}{\mathbb{N}}
\newcommand{\ZZ}{\mathbb{Z}}
\newcommand{\QQ}{\mathbb{Q}}
\newcommand{\RR}{\mathbb{R}}

\newcommand{\TT}{\mathbb{T}}

\begin{document}

\title[]{On the existence of periodic invariant curves 
for analytic families of twist maps and billiards}

\author{Corentin Fierobe}
\address{Institute of Science and Technology Austria, Am Campus 1, 3400 Klosterneuburg, Austria}
\email{corentin.fierobekoz@gmail.com}

\author{Alfonso Sorrentino}
\address{Department of Mathematics, University of Rome Tor Vergata, Via della ricerca scientifica 1, 00133 Rome, Italy}
\email{sorrentino@mat.uniroma2.it}

\maketitle 

\begin{abstract} 
In this paper we prove that in any analytic one-parameter family of twist maps of the annulus, homotopically invariant curves filled with periodic points corresponding to a given rotation number, either exist for all values of the parameters or at most for a discrete subset. Moreover, we show that the set of analytic twist maps having such an invariant curve of a given rotation number is a strict analytic subset of the set of analytic twist maps. The first result extends, in dimension $2$, a previous result by Arnaud, Massetti and Sorrentino \cite{AMS}. We then apply our result to rational caustics of billiards, considering several models such as Birkhoff billiards, outer billiards and symplectic billiards.
\end{abstract}


\section{Introduction}
The study of invariant manifolds in Hamiltonian systems is a significant area of research, stemming from the pioneering works of Poincaré and subsequent seminal contributions by Kolmogorov, Arnold, and Moser, which  laid the groundwork for what is now known as KAM theory \cite{Kolmogorov, Arnold, Moser}). While the investigation of the existence of invariant manifolds is crucial for understanding the stability of these systems, their destruction holds pivotal importance in elucidating the transition from stability to instability, from integrable regimes to non-integrable ones.

In this article, we focus on twist maps of the annulus and the existence of invariant curves. Of particular interest are invariant curves foliated by periodic points, which emerge as delicate structures within the dynamical landscape. Understanding the essence of this fragility plays a crucial role in trying to tackle several  of the foremost questions and conjectures in dynamics (see for example, recent developments around integrable billiards and Birkhoff's conjecture \cite{Bialy, KaloshinKoudjinan, KaloshinSorrentino, Koval}).

More specifically, we consider real-analytic one-parameter families of exact twist maps (see Definition \ref{definition:general_symplectic_twist_map} for a more precise statement) and investigate the topological structure of the subset of parameters corresponding to maps admitting an invariant curve foliated by periodic points and with a given rational rotation number. 
Our main result can be  summarized as follows:


{\textbf{Main Theorem 1} (see Theorem \ref{theorem:main2} for a more precise statement). \textit{Given an interval $J$ and a real-analytic family of exact symplectic twist maps $(F_{\varepsilon})_{\varepsilon\in J}$, the set of parameters $\varepsilon\in J$ such that $F_{\varepsilon}$ admits an invariant curve filled with periodic points of a given rotation number is either discrete or the whole interval $J$.}}

This result generalizes to arbitrary real-analytic families of twist-maps of the annulus a previous result by Arnaud, Massetti and Sorrentino \cite{AMS}, which considered special families generated by perturbation by a potential, in any dimension.

Actually, one could prove a more general version of Main Theorem 1 that does not involve the parameter $\varepsilon$. In fact, one can show that the set of twist maps having an invariant curve filled with periodic points of a given rotation number has the structure of an analytic set; this means that it is closed and it is locally given by the zeros of an analytic map.

{\textbf{Main Theorem 2} (see Theorem \ref{theorem:main_Banach} for a more precise statement).\textit{ The set of real-analytic exact symplectic twist-maps having an invariant curve filled with periodic points of a given rotation number is a strict analytic subset of the set of real-analytic exact symplectic twist-maps.\\}}

Our main motivation for these results came from billiard dynamics. Therefore,  we apply the above theorems to several billiard models: classical Birkhoff billiards, outer billiards and symplectic billiards.\\

The article is organized as follows:
\begin{itemize}
\item In {\bf Section \ref{secpreliminary}} we provide a more detailed description of the setting, so to state more precisely our main theorems, see \textbf{Theorems \ref{theorem:main2}} and {\bf \ref{theorem:main_Banach}}. \\ Moreover, we present the application of the main theorems  to the various billiard models; more specifically, \textbf{Subsections \ref{subsection:intro_inner_billiard}} for \textit{classical Birkhoff billiards}, \textbf{Subsection \ref{subsection:intro_outer_billiard}} for \textit{outer billiards} and \textbf{Subsection \ref{subsection:intro_symplectic_billiard}} for \textit{symplectic billiards}.
\item In \textbf{Section \ref{section:invariant_graphs}} we recall and prove some properties of periodic invariant graphs.
\item \textbf{Section \ref{section:proof_main}} is devoted to the proof of the Main Theorem (Theorem \ref{theorem:main2}). 
\item In \textbf{Sections \ref{section:proof_inner_billiard}}, \textbf{\ref{section:proof_outer_billiard}} and \textbf{\ref{section:proof_symplectic_billiard}} we discuss the proofs of the results related to the various billiard models.
\item Finally, in \textbf{Section \ref{section:proof_main_Banach}} we prove Main Theorem 2 and its application to Birkhoff Billiards ({\bf Theorems \ref{theorem:main_Banach}} and {\bf \ref{theorem:main_Banach_billiard}}).
\end{itemize}

\section{Acknowledgments}

The authors are grateful to the Simons Center for Geometry and Physics at Stony Brook for its hospitality and support during the program ``{\it Mathematical billiards: at the crossroads dynamics, geometry, analysis, and mathematical physics}'', where part of this project was carried out.\\
AS also acknowledges the support of the Italian Ministry of University and Research’s  PRIN 2022 grant ``{\it Stability in Hamiltonian dynamics and beyond}’', as well as the Department of Excellence grant MatMod@TOV (2023-27) awarded to the Department of Mathematics of University of Rome Tor Vergata. AS is a member of the INdAM research group GNAMPA and the UMI group DinAmicI.\\

\medskip

\section{Preliminaries and statements of the  results} \label{secpreliminary}
In this section we provide a more precise description of the setting and state the main results and their applications to billiard models, in full details.

\subsection{Exact-symplectic twist maps}

On the space $\RR^{2}$ of pairs $(p,q)$, consider $\pi_q,\pi_p:\RR^{2}\to\RR$ the projections, respectively, onto $q$ and $p$.

Consider two continuous $\ZZ$-periodic maps $p^-, p^+:\RR\to\RR\cup\{\pm\infty\}$ and assume that the inequality $p^-(q)<p^+(q)$ is satisfied for any $q\in\RR$. We define the open strip
$$\mathbb A_{p^{\pm}} = \{(q,p)\in\RR^{2}\,|\,p^-(q)<p<p^+(q)\}.$$
It is a bundle over $\RR$, whose fibers are the intervals
$$\restrict{\mathbb A_{p^{\pm}}}{q} = \{p\in\RR\,|\,p^-(q)<p<p^+(q)\}.$$
$\mathbb A_{p^{\pm}}$ projects onto an interval bundle over the torus $\TT^1:=\RR/\ZZ$ having the same fibers. We denote by $\graph(p^-)$ and  $\graph (p^+)$, respectively, the graphs of $p^-$ and $p^+$ in $\RR^{2}$. 


\begin{definition}
\label{definition:general_symplectic_twist_map}
A diffeomorphism $F:  \mathbb A_{p^{\pm}} \longrightarrow \mathbb A_{p^{\pm}}$, where $F(q,p):=(Q(q,p),P(q,p))$, is called an \textit{exact-symplectic twist map} if it satisfies the following properties:
\begin{itemize}
\item[{\bf (i)}] (\textit{Periodicity}) $F(q+m,p)=F(q,p)+(m,0)$ for any $(q,p)\in\mathbb A_{p^{\pm}}$ and $m\in\ZZ$;
\item[{\bf (ii)}]  (\textit{Twist condition}) for any $q\in\RR$ the map 
$$p\in\restrict{\mathbb A_{p^{\pm}}}{q} \longmapsto Q(q,p)$$
is a diffeomorphism onto its image;
\item[{\bf (iii)}] (\textit{Boundary preservation}) For any neighborhood $V$ of $\graph(p^-)\cup \graph(p^+)$ in $\RR^{2}$, there exists another  neighborhood $U$ satisfying
$$F(U\cap\mathbb A_{p^{\pm}})\subseteq V\cap\mathbb A_{p^{\pm}};$$
\item[{\bf (iv)}]  (\textit{Generating function)} There is an open set $\mathcal D\subseteq\RR^{2}$ and a smooth map $S:{\mathcal D}\to\RR$, called \textit{generating function} of $F$, such that for any $(q,Q)\in{\mathcal D}$ and $m\in \ZZ$:
\begin{eqnarray*}
&& \quad (q,Q)\in{\mathcal D} \quad \Longrightarrow\quad (q+m,Q+m)\in{\mathcal D} \\
&& \quad S(p+m,Q+m) = S(q,Q),
\end{eqnarray*}
and 
$$PdQ-pdq=dS(q,Q).$$
\end{itemize}
\end{definition}

\medskip

\begin{remark}
Given an exact symplectic twist map $F$, the map $p\in\restrict{\mathbb A_{p^{\pm}}}{q} \mapsto Q(q,p)$ is either strictly increasing for any $q$, or strictly decreasing for any $q$. In the first case, we say that $F$ is \textit{positive}, in the second one that it is \textit{negative}. \\
Note also that since 
$\partial_{12}^2 S = -(\partial_p Q)^{-1}$ \cite[Formula (9.2.4)]{KatokHasselblatt}, we observe that $\partial_{12}^2 S<0$ if $F$ is positive, and $\partial_{12}^2 S>0$ if $F$ is negative. In the proofs, we will often assume that $F$ is positive to simplify the redaction, since the proofs in the negative case are analogous.
\end{remark}

\begin{remark}
An exact symplectic twist map $F$ induces a map $f$ from the tangent bundle $T\TT^1\simeq\TT^1\times\RR$ to itself. More precisely, if $\pi:\RR^2\to\TT^1\times\RR$ is the canonical projection, then $f:\pi(\mathbb A_{p^{\pm}})\to\pi(\mathbb A_{p^{\pm}})$ is defined by $f\circ\pi = \pi\circ F.$ 
In order to ease notation, we will use the same notations for both maps.
\end{remark}

\subsection{Periodic and invariant graphs}

A \textit{rotational invariant curve} of a symplectic twist map $F:\mathbb A_{p^{\pm}}\to\mathbb A_{p^{\pm}}$ is a curve $\Gamma\subset \mathbb A_{p^{\pm}}$ such that $F(\Gamma)=\Gamma$ and   $\mathbb A_{p^{\pm}} \setminus \Gamma$  consists of two connected components. 

\begin{remark} \label{remark:Birkhoff_theorem} 
A famous theorem by Birkhoff (see for example \cite[Theorem 15.1]{ForniMather}) states that any rotational invariant curve $\Gamma$, if it exists, is the graph of a Lipschitz continuous $1$-periodic map $\gamma:\RR\to\RR$. Moreover, the Lipschitz constant of $\gamma$ only depends on $\inf_{\Gamma} \partial_pQ>0$ (see \cite[Theorem 15.1]{ForniMather},  \cite[Lemma 13.1.1]{KH_book}  and \cite[Proposition 12.3]{Gole} for more details). 
\end{remark}

This leads to the following definition.

\begin{definition}
Let $F:\mathbb A_{p^{\pm}}\to\mathbb A_{p^{\pm}}$ be an exact-symplectic twist map, $\Gamma=\text{graph}(\gamma)\subset\mathbb A_{p^{\pm}}$ be the graph of a $1$-periodic Lipschitz-continuous map $\gamma:\RR\to\RR$, and $m \in \ZZ ,n\in\NN^*$ coprime. We say that $\Gamma$ is
\begin{itemize}
\item[{\bf (i)}] \textit{$(m,n)$-periodic} if $F^n(q,\gamma(q))=(q+m,\gamma(q))$ for any $q\in\RR$;
\item[{\bf (ii)}] \textit{invariant by} $F$ if $F(\Gamma)\subseteq\Gamma$;
\item[{\bf (iii)}] \textit{$\class^k$-smooth} (resp. {\it analytic}) if $\gamma$ is a $\class^k$-smooth (resp. analytic).
\end{itemize}
\end{definition}

We show in Proposition \ref{proposition:equivalence_invariance_minimizing} that $(m,n)$-periodic graphs are automatically invariant by $F$ and have the same regularity as $F$.


\subsection{Twist interval} Given an $(m,n)$-periodic invariant graph $\Gamma$ of $F$, the restrictions of $\restrict{F}{\Gamma}$ can be seen as {the lift of a diffeomorphism of $\TT^1$} whose rotation number is $m/n$. We introduce the set of all possible rotation numbers of orbits of $F$, and call it \textit{twist interval} of $F$. It is defined as follows (see also \cite[Definition 9.3.2]{KH_book}):

\begin{definition} 
The \textit{twist interval} of a symplectic twist map $F$ is the set  $\text{TI}(F)$ of numbers $\alpha\in\RR$ for which there is a neighborhood $U^-$ of $\graph(p^-)$ and a neighborhood $U^+$ of $\graph(p^+)$ in $\RR^{2}$ such that for $(q,p)\in\mathbb A_{p^{\pm}}$ 
$$(q,p)\in U^-\;\Longrightarrow\;\pi_q\circ F(q,p)-q\leq \alpha$$ 
and 
$$(q,p)\in U^+\;\Longrightarrow\;\pi_q\circ F(q,p)-q \geq \alpha.$$
\end{definition}

\medskip

\begin{remark}
{\bf (i)} The twist interval is by construction an open interval of $\RR$. For example, given $\varepsilon\in\RR$, the map $F_{\varepsilon}:\RR\times(0,1)\to\RR\times(0,1)$ defined for all $(q,p)\in\RR\times(0,1)$ by
$F(q,p)=(q+p+\varepsilon,p)$ is an exact symplectic twist map whose twist interval is $\text{TI}(F_{\varepsilon})=(\varepsilon,1+\varepsilon).$\\
{\bf (ii)}
Another example can be given in the case of the usual billiard map (this will discussed in more details in Section \ref{subsection:intro_inner_billiard}). Given a strictly convex domain $\Omega$ (fix an orientation of its boundary $\partial \Omega$), the billiard map $F_{\Omega}:\RR\times(-1,1)\to\RR\times(-1,1)$ is given by $F_{\Omega}(q,-\cos\varphi)=(q_1,-\cos\varphi_1)$ where $(q_1,\varphi_1)$ is the pair describing the point of impact and the angle of reflection after the bounce of a trajectory coming from $q$ and making an angle $\varphi$ with the (oriented) tangent vector at $\Omega$ in $q$. The twist interval of $F_{\Omega}$ is given by $TI(F_{\Omega}) = (0,1)$: indeed, when $\varphi$ goes from $0$ to $\pi$, the point of impact $q_1$ moves along the boundary $\partial\Omega$ from $q$ to itself, winding exactly once around $\partial\Omega$.
\end{remark}


\subsection{Main theorem 1}

We now consider  {analytic} one-parameter families of symplectic twist maps. More specifically, consider an interval $I\subset\RR$ and continuous maps $p^-,p^+:I\times\RR\to\RR\cup\{\pm\infty\}$ that are $1$-periodic in the second component and such that
for any $(\varepsilon,q)\in I\times\RR$ the inequality $p^-(\varepsilon,q)<p^+(\varepsilon,q)$.  \\
One can introduce the open set
$$\mathbb A_{I,p^{\pm}} := \{(\varepsilon,q,p)\in I\times\RR^{2}\,|\,p^-(\varepsilon,q)<p<p^+(\varepsilon,q)\}$$
and denote its closure by
$$\overline{\mathbb A}_{I,p^{\pm}} = \{(\varepsilon,q,p)\in I\times\RR^2\,|\,p^-(\varepsilon,q)\leq p\leq p^+(\varepsilon,q)\}.$$
Given $\varepsilon\in I$, we denote its {$\varepsilon$-}section by
$$\mathbb A_{I,p^{\pm}}^{\varepsilon} := \{(q,p)\in \RR^{2}\,|\,p^-(\varepsilon,q)<p<p^+(\varepsilon,q)\}.$$

\medskip


\begin{theorem}[Main Theorem 1]
\label{theorem:main2}
Assume that $I\subset\RR$ is {an interval} and $(m,n)\in\ZZ\times\NN^{\ast}$ are coprime. Suppose that for any $\varepsilon\in I$ we are given an exact symplectic twist map $F_{\varepsilon}$ such that:
\begin{itemize}
\item[{\bf (i)}]  the map $(\varepsilon,q,p)\in \mathbb A_{I,p^{\pm}}\mapsto F_{\varepsilon}(q,p)$ is analytic;
\item[{\bf (ii)}] $m/n \in TI(F_{\varepsilon})$ for every $\varepsilon\in I$.
\end{itemize}
Then,  the set 
$$\mathcal{I}_{(m,n)}(\RR):=\left\{\varepsilon\in I\,|\, F_{\varepsilon} \text{ has an }(m,n)\text{-periodic invariant graph}\right\}$$
{is either discrete or consists of the whole $I$}. In particular, if $I$ is compact, then $\mathcal{I}_{(m,n)}(\RR) = I$ or it is at most finite.\\
\end{theorem}

\begin{remark}
In the statement of Theorem \ref{theorem:main2}, we do not need to precise the regularity of invariant graphs. In fact, it follows from Proposition \ref{proposition:equivalence_invariance_minimizing} that they are indeed {analytic}.  \\
\end{remark}

\subsection{Main Theorem 2}
\label{subsection:main_Banach}
Theorem \ref{theorem:main2} can be generalized without involving the parameter $\varepsilon$. 
 Let us define the set $\mathscr T^{\omega}$ of analytic twist-maps of $\mathbb A_{p^{\pm}}$ which admit a continuous extension to the closure $\overline{\mathbb A_{p^{\pm}}}$. It is an open subset of the Banach set of analytic maps of $\mathbb A_{p^{\pm}}$ to itself admitting a continuous extension to $\overline{\mathbb A_{p^{\pm}}}$.

Before stating our  result, let us recall the definition of analytic subset of a Banach space. We refer the reader to \cite{Whittlesey} for more details about this topic.

\begin{definition}\label{defbanach}
Let $U$ be an open set of a Banach space $E$ and let $X\subset U$ be a closed subset. 
\begin{itemize}
\item $U$ is said an \textit{analytic subset} of $U$ if for any $x\in U$, there is an analytic map $f_x:V_x\to F$ defined on a neighborhood $V_x\subset U$ of $x$ to a Banach space $F$ and such that $X\cap V_x=f_x^{-1}(0)$. 
\item $X$ is said to be a \textit{strict analytic subset}  if $X\neq U$. 
\end{itemize}
\end{definition}

\begin{remark}
If $U$ is connected, strict analytic subsets of $U$ must have empty interior.
\end{remark}

We can now state our result.

\begin{theorem}[Main Theorem 2]
\label{theorem:main_Banach}
Let $m/n\in\QQ$. The subset of twist-maps of $\mathscr T^{\omega}$ having an $(m,n)$-periodic invariant graph is a strict analytic subset of $\mathscr T^{\omega}$.\\
\end{theorem}

\begin{remark}\label{remaftermainthm2}
One could actually deduce the proof of Theorem \ref{theorem:main2} using the statement of Theorem \ref{theorem:main_Banach}. In fact, if $\varphi$ is the analytic map $\varphi:\varepsilon\in I\mapsto F_{\varepsilon}\in \mathscr T^{\omega}$ and $\mathscr T^{\omega}_{m,n}$ is the subset of twist-maps of $\mathscr T^{\omega}$ having an $(m,n)$-periodic invariant graph, then $\mathcal I_{(m,n)}(\RR)$ is simply $\mathcal I_{(m,n)}(\RR)=\varphi^{-1}(\mathscr T^{\omega}_{m,n})$, hence is an analytic subset of $I$. It follows that it is either the whole $I$ or it is discrete (this property on analytic subsets of $\RR$ holds also with our definition of analytic subset). However, we decided to provide an independent proof of Theorem \ref{theorem:main2}, which is closer to our initial motivations.\\
\end{remark}

\subsection{Application to Birkhoff billiards}
\label{subsection:intro_inner_billiard}
A classical (planar) billiard is a bounded domain $\Omega\subset\RR^2$ with (piecewise) smooth boundary, in which one can study the behaviour of an infinitely small particle evolving inside $\Omega$ without friction. When reaching the boundary, the particle bounces on it according to the usual reflection law of geometrical optics: \textit{the angle of incidence equals the angle of reflection}. We refer the interested reader to \cite{KozlovTreshchev,TabachnikovBook} for more details on billiards, and to \cite[Chapter 3]{Siburg} for an overview on rigidity questions in billiards.

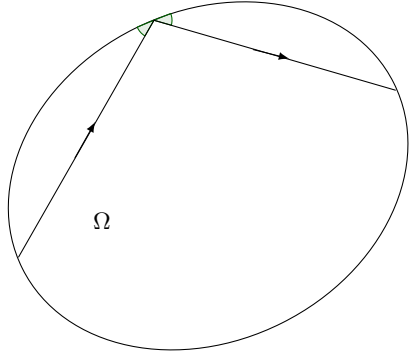
\begin{figure}[!h]
\definecolor{qqwuqq}{rgb}{0,0.39,0}
\definecolor{uququq}{rgb}{0.25,0.25,0.25}
\definecolor{qqqqff}{rgb}{0,0,1}
\definecolor{xdxdff}{rgb}{0.49,0.49,1}
\definecolor{cqcqcq}{rgb}{0.75,0.75,0.75}
\begin{tikzpicture}[line cap=round,line join=round,>=triangle 45,x=3.0cm,y=3.0cm]
\clip(-1.4,-0.6) rectangle (1,1.4);
\draw [rotate around={26.57:(0,0.25)}] (0,0.25) ellipse (2.76cm and 2.19cm);
\draw [shift={(-0.24,0.94)},color=qqwuqq,fill=qqwuqq,fill opacity=0.1] (0,0) -- (-157.88:0.08) arc (-157.88:-119.9:0.08) -- cycle;
\draw [shift={(-0.24,0.94)},color=qqwuqq,fill=qqwuqq,fill opacity=0.1] (0,0) -- (-15.81:0.08) arc (-15.81:22.12:0.08) -- cycle;
\draw (-0.84,-0.11)-- (-0.24,0.94);
\draw (-0.24,0.94)-- (0.83,0.63);
\draw [-latex] (-0.59,0.33) -- (-0.5,0.49);
\draw [-latex] (0.2,0.81) -- (0.36,0.77);
\begin{scriptsize}
\draw[color=black] (-0.47,0.05) node {$\Omega$};
\draw[color=black] (-0.28,1.02) node {$ $};
\end{scriptsize}
\end{tikzpicture}
\caption{The classical reflection law of a particle inside a strictly convex domain $\Omega$ with smooth boundary.}
\end{figure}

Let us consider the case when $\Omega$ is strictly convex with a smooth boundary, which defines a so-called \textit{Birkhoff billiard}. The dynamics of a particle in $\Omega$ is described by a discrete map, \textit{the billiard map}, acting on the space of oriented lines intersecting $\Omega$: given  a line $\ell$, the billiard map associates to it the line $\ell'$ naturally obtained by reflecting $\ell$ at the point of impact with $\partial\Omega$ (see Figure 1). The phase space is a cylinder, which can be parametrized by pairs $(s,\varphi)\in\RR/|\partial\Omega|\ZZ\times[0,\pi]$, where $s$ is an arc-length coordinate on the boundary $\partial\Omega$ and $\varphi$ is the angle between the tangent line of $\Omega$ at $s$ and the corresponding oriented line starting at $s$.

As for dynamical systems in general, one can study the so-called \textit{integrable billiards}: billiards whose phase space contains an open set foliated by curves which are invariant by the billiard map. Circles are examples of such billiards for which the whole phase space is foliated by invariant curves -- in this case we speak about \textit{globally} integrable billiards. Let us mention also the case of ellipses, which are integrable, but not globally integrable. A famous conjecture, due to Birkhoff \cite{Birkhoff} and Poritsky \cite{Poritsky}, states that

\begin{conjecture}[Birkhoff-Poritsky]
The only integrable billiards are ellipses. 
\end{conjecture}

Bialy \cite{Bialy} showed that the only globally integrable billiards are circles. Kaloshin-Sorrentino \cite{KaloshinSorrentino} proved that the only billiards close to ellipses having invariant curves of rotation number $1/q$ for any $q\geq 2$ are ellipses. Later, Koval \cite{Koval} extended this result to billiards close to generic ellipses having invariant curves of rotation number $r=p/q$, for any $r$ lower than an arbitrarily small bound of the form $1/q_0$. The \textit{rotation number} of an invariant curve is defined as the rotation number of the circle map obtained by restricting the billiard map to the corresponding curve.

In the case of non-rational rotation number, Lazutkin \cite{Lazutkin_KAM} showed that there is a Cantor set $C\subset[0,1]$ of non-zero measure accumulating to $0$ such that each $\omega\in C$ is the rotation number of an invariant curve. Popov showed \cite{Popov} that these curve persists under a small deformation of the billiard. However the rotation numbers considered in these results are far from being rational: they are so-called Diophantine numbers, which are numbers badly approximated by rationals.

In the case of rational rotation numbers, it is expected that corresponding invariant curves are more fragile. 

The main result in this paper allows us to prove the following result.

\begin{theorem}[Invariant curves in families of Birkhoff billiards]
\label{theorem:birkhoff_billiard}
Let $I$ be an interval and $(\Omega_{\varepsilon})_{\varepsilon\in I}$ be an analytic family of strictly convex analytic domains. Then given a pair $(m,n)\in\ZZ\times\NN^{\ast}$ of coprime integers such that $m/n\in(0,1)$, the set of $\varepsilon\in I$ such that the billiard map inside $\Omega_{\varepsilon}$ has an $(m,n)$-periodic invariant curve is either discrete or consists of the whole $I$. 
\end{theorem}

This result does not answer Birkhoff's conjecture, but rather confirms how fragile integrability is.  
Let us also mention the work of Kaloshin-Koudjinan \cite{KaloshinKoudjinan} in which they study for billiard domains close to a disk the coexistence of two invariant curves of rotation numbers $1/2$ and $1/2q+1$ for $q\geq 1$.\\

The proof  Theorem \ref{theorem:birkhoff_billiard}  is a consequence of Main Theorem (Theorem \ref{theorem:main2}), as it will be shown in Section \ref{section:proof_inner_billiard}.\\

\begin{remark}
An interesting application of Theorem \ref{theorem:birkhoff_billiard} regards analytic families of  Birkhoff billiards generated by the action of a geometric flow.  One of the main motivations comes from  \cite[Appendix F]{KaloshinSorrentino}, where  the authors proposed a possible approach to extend their local analysis (near ellipses) to a more global one, by 
considering the evolution  of a given Birkhoff billiard under the so-called \textit{affine curvature flow}: the family of billiards thus obtained will ``converge'' (up to renormalize their area) to ellipses (see for example \cite{SapiroTannenbaum}  for more details).

More precisely, given a strictly convex bounded domain $\Omega$ with analytic boundary and $\delta>0$, there is a one parameter family $(\Omega_{\varepsilon})_{\varepsilon\in[0,T]}$ of strictly convex bounded domains satisfying the following properties:
\begin{itemize}
\item $\Omega_0=\Omega$ and $\Omega_T$ is $\delta$-close to an ellipse;
\item for $\varepsilon>0$, the domain  $\Omega_{\varepsilon}$ has an analytic boundary, and the family $(\Omega_{\varepsilon})_{\varepsilon\in(0,T]}$ is analytic in $\varepsilon$.
\end{itemize}

The crucial observation is that $\Omega_{\varepsilon}$ are ellipses if and only if one of them is an ellipse.\\
One of the main obstacle in applying this idea is that it is not easy to show that integrability is preserved by the action of the flow: this statement turns out to be equivalent to Birkhoff conjecture.\\
An immediate application of Theorem \ref{theorem:birkhoff_billiard} implies that either the family consists of all integral billiards $(\Omega_{\varepsilon})_{\varepsilon\in[0,T]}$ or there are at most finitely many elements of the family that are integrable.\\
\end{remark}

Following the same approach as in Theorem \ref{theorem:main_Banach}, we can generalize the statement of Theorem \ref{theorem:birkhoff_billiard} in the following way.

In the set of domains with analytic boundary, we distinguish those whose boundary has a strip of analyticity of at least a given size, so to be able to consider Banach spaces of domains. 

It is well-known that, up to translations and rotations, strictly convex domains of $\RR^2$ with $\mathscr C^{k+1}$-smooth boundary is in one-to-one correspondence with the set of $\mathscr C^k$-smooth $1$-periodic maps $\varrho:\RR\to\RR^{>0}$ satisfying
\begin{equation}
\label{equation:0_Fourier}
\int_0^1\varrho(\theta)e^{2i\pi\theta}d\theta=0.
\end{equation}
The idea is that given such a map $\varrho$, the function $\gamma$ defined for any $\theta\in[0,2\pi]$ by
$$\gamma(\theta) := \left(\int_0^{\theta}\varrho(t)\cos(t)dt,\int_0^{\theta}\varrho(t)\sin(t)dt\right)$$
parametrizes a simple closed curve containing the origin and which is tangent to the $x$-axis at the origin ($\theta$ represents the angle that the tanget to the curve forms with the positive $x$ semi-axis). This idea is developped in more details in \cite[Section 2]{MarviziMelrose}.

Hence, given $r>0$ we consider the set $\mathscr D_r$ of $1$-periodic analytic maps $\varrho:\RR\to\RR^{>0}$ satisfying \eqref{equation:0_Fourier} and having a strip of analyticity of size at least $r$, in the following sense: if the Fourier expansion of $\varrho$ is $\varrho(\theta):=\sum_{p\in\ZZ}\widehat\varrho_pe^{2i\pi p\theta}$, then we assume that $\widehat\varrho_p = \mathcal O(e^{-r|p|})$ as $p\to\pm\infty$. $\mathscr D_r$ is a Banach space.
Thus we can consider analytic subsets of $\mathscr D_r$, as introduced in Definition \ref{defbanach} and prove the following theorem.

\begin{theorem}
\label{theorem:main_Banach_billiard}
Let $m/n\in(0,1)$. The set of domains $\Omega\in\mathscr{D}_r$ such that the billiard map in $\Omega$ has an $(m,n)$-periodic invariant graph is a strict analytic subset of $\mathscr{D}_r$.
\end{theorem}

\begin{remark}
Similarly to what observed in Remark \ref{remaftermainthm2}, one could use Theorem \ref{theorem:main_Banach_billiard} to obtain an alternative proof of Theorem \ref{theorem:birkhoff_billiard}.
\end{remark}

\subsection{Application to dual billiards} 
\label{subsection:intro_outer_billiard}
Given a strictly convex domain $\Omega\subset\RR^2$ with a smooth oriented boundary, the \textit{dual} or {\it outer billiard outside $\Omega$} can be defined as follows (see Figure \ref{figure:dual_billiard}). For any point $p\in\RR^2\setminus\Omega$, there are at most two tangent lines to $\partial\Omega$ passing through $p$. Consider the unique one which is tangent to $\partial\Omega$ at a point $q$ and such that the vector $\vec{pq}$ has the same orientation as the boundary $\partial\Omega$ at $q$. Define the image by $p$ by the dual billiard map as the point $F(p)$ on the latter tangent line $T_q\partial\Omega$ such that $q$ is the midpoint between $p$ and $F(p)$ (see Figure 2).

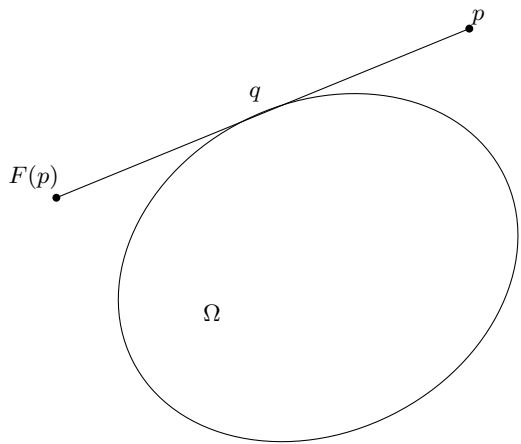
\begin{figure}[!h]
\definecolor{uququq}{rgb}{0.25,0.25,0.25}
\definecolor{qqqqff}{rgb}{0,0,1}
\definecolor{xdxdff}{rgb}{0.49,0.49,1}
\definecolor{cqcqcq}{rgb}{0.75,0.75,0.75}
\begin{tikzpicture}[line cap=round,line join=round,>=triangle 45,x=3.0cm,y=3.0cm]
\clip(-1.4,-0.6) rectangle (1,1.4);
\draw [rotate around={26.57:(0,0.25)}] (0,0.25) ellipse (2.76cm and 2.19cm);
\draw (-1.16,0.56)-- (0.67,1.31);
\begin{scriptsize}
\draw[color=black] (-0.47,0.05) node {$\Omega$};
\draw[color=black] (-0.28,1.02) node {$q$};
\fill [color=black] (-1.16,0.56) circle (1.5pt);
\draw[color=black] (-1.26,0.65) node {$F(p)$};
\fill [color=black] (0.67,1.31) circle (1.5pt);
\draw[color=black] (0.71,1.36) node {$p$};
\end{scriptsize}
\end{tikzpicture}
\caption{The point $F(p)$ is the image of $p$ by the dual billiard map around the domain $\Omega$. The point $q$ is the midpoint between $p$ and $F(p)$ which are supported by a line tangent to $\partial\Omega$ at $q$.}
\label{figure:dual_billiard}
\end{figure}

{According to \cite{MoserBook}, dual billiards were introduced by B. H. Newman in 1960 and also mentionned by P. C. Hammer in \cite{Hammer} a bit later.} Their properties were largely studied since then. They are known to be symplectic twist maps of the infinite annulus $\TT^1\times(0,+\infty)$, and thus can be investigated in the context of Aubry-Mather theory \cite{Boyland}. Douady \cite{Douady} showed that if the boundary $\partial\Omega$ is sufficiently smooth (at least $\mathcal C^6$), then there is a positive measure set of invariant curves accumulating to the boundary, as well as a positive measure set of invariant curves accumulating at infinity. In paticular this gives a negative answer to the famous question of the existence of unbounded orbits. Adapting a result of Mather for Birkhoff billiards (see \cite{ForniMather}) Boyland \cite{Boyland} proved that if the curvature of a convex  domain vanishes, or has jump discontinuities, then there is a neighborhood of the boundary without invariant curves. 

A version of Birkhoff's conjecture is also studied for dual billiards. It is indeed known that the phase space of dual billiards around ellipses is foliated by invariant curves, induced by any bigger ellipse homothetically equivalent to the initial billiard.  Bialy \cite{BialyOuter} proved a total integrability result: if the phase space of a dual billiard is foliated by continuous invariant curves then the billiard is an ellipse. If we assume the foliation to be only in a open set of the phase space (local integrability), some partial positive results were given in \cite{GlutsyukShustin, TabachnikovOuter}.

In this context, our main result reads:\\

\begin{theorem}[Invariant curves in families of outer billiards]
\label{theorem:outer_billiard}
Let $I$ be an interval and $(\Omega_{\varepsilon})_{\varepsilon\in I}$ be an analytic family of strictly convex analytic domains. Then given a pair $(m,n)\in\ZZ\times\NN^{\ast}$ of coprime integers such that $m/n\in(0,1)$, the set of $\varepsilon\in I$ such that the outer billiard map associated to $\Omega_{\varepsilon}$ has an $(m,n)$-periodic invariant curve is either discrete or consists of the whole $I$. 
\end{theorem}

Theorem \ref{theorem:outer_billiard} for outer billiards is a consequence of Main Theorem (Theorem \ref{theorem:main2}), as it will be shown in Section \ref{section:proof_outer_billiard}.

{
\subsection{Application to symplectic billiards} 
\label{subsection:intro_symplectic_billiard} As for classical billiards, {\it symplectic billiards} -- which were firstly introduced by Albers and Tabachnikov in \cite{AlbersTabachnikov} -- are defined inside a strictly convex bounded domain $\Omega$ with smooth boundary. They describe the evolution of a infinitesimally small ball inside $\Omega$ which bounces on the boundary $\partial\Omega$ according to the following reflection law: given three successive impact points $p_1$, $p_2$ and $p_3$, the line joining $p_1p_3$ and the tangent line $T_{p_2}\partial\Omega$ of $\partial\Omega$ at $p_2$ are parallel (see Figure 3). We refer to \cite{AlbersTabachnikov} for more details and results related to this billiard model.

\begin{figure}{figsimplbill}
\centering
\begin{tikzpicture}[line cap=round,line join=round,>=triangle 45,x=2.0cm,y=2.0cm]
\clip(-2,-1.2) rectangle (2,1.35);
\draw [rotate around={0:(0,0)}] (0,0) ellipse (2.82cm and 2cm);
\draw [dash pattern=on 1pt off 1pt,domain=-3.19:3.33] plot(\x,{(--32-8.43*\x)/29.69});
\draw [dash pattern=on 1pt off 1pt,domain=-3.19:3.33] plot(\x,{(-10.52-8.43*\x)/29.69});
\draw [-latex] (1.07,-0.66) -- (0.72,0.37);
\draw [-latex] (0.53,0.93) -- (-0.58,0.43);
\draw (-0.58,0.43)-- (-1.41,0.05);
\draw (0.72,0.37)-- (0.53,0.93);
\begin{scriptsize}
\draw[color=black] (1.15,-0.77) node {$p_1$};
\draw[color=black] (0.56,1.02) node {$p_2$};
\draw[color=black] (-1.5,-0.05) node {$p_3$};
\draw[color=black] (1.5,0.8) node {$T_{p_2}\partial\Omega$};
\draw[color=black] (0,0) node {$\Omega$};
\end{scriptsize}
\end{tikzpicture}
\caption{A symplectic billiard bounce $(p_1,p_3)\mapsto (p_2,p_3)$ in a strictly convex domain $\Omega$.}
\end{figure}

The symplectic billiard map $(p_1,p_2)\mapsto(p_2,p_3)$ describes the symplectic billiard trajectories inside $\Omega$. It appears \cite{AlbersTabachnikov} that there exists  a set of coordinates in which this map can be expressed as an exact symplectic twist map whose generating map is related to the canoncial symplectic area form in the plane of coordinates $(x,y)$, namely $dx\wedge dy$ - whence the name of this billiard model. It was shown, see \cite[Theorem 3]{AlbersTabachnikov}, that the latter map has a positive measure set of invariant curves, which in fact are caustics. The latter however have highly irrational rotation numbers, and are not concerned by our result.

Still, we show that Theorem \ref{theorem:main2} can be applied to rational invariant curves in analytic families of symplectic billiards; namely:
\begin{theorem}[Invariant curves in families of symplectic billiards]
\label{theorem:symplectic_billiard}
Let $I$ be an interval and $(\Omega_{\varepsilon})_{\varepsilon\in I}$ be an analytic family of strictly convex analytic domains. Then given a pair $(m,n)\in\ZZ\times\NN^{\ast}$ of coprime integers such that $m/n\in(0,1)$, the set of $\varepsilon\in I$ such that the symplectic billiard map associated to $\Omega_{\varepsilon}$ has an $(m,n)$-periodic invariant curve is either finite or consists of the whole $I$. 
\end{theorem}

\subsection{Other billiard models} In previous subsections, we described how Theorem \ref{theorem:main2} can be applied to different billiard models such as the classical inner billiards, outer billiards and symplectic billiards. Actually Theorem \ref{theorem:main2} could be applied to many other billiard models, as long as their dynamics can be encoded by an exact-symplectic twist map. 

For a list of billiard models, the interested reader can consult \cite{AlbersTabachnikov_Dowker}. As an example, let us mention the outer length billiard, which is an analogue of the outer billiard model, but with a different generating function. In fact, periodic orbits are given by extremizing not the area of a circumscribed polygon -- as for outer billiards, but its perimeter. In this setting, the two pairs (classical billiards, symplectic billiards) and (outer length billiards, outer billiards) can be seen as ``dual'' to each other.}

\section{Preliminary results on periodic graphs}
\label{section:invariant_graphs}

In this section we prove some preliminary results on invariant periodic graphs of a symplectic twist map $F:\mathbb A_{p^{\pm}}\to \mathbb A_{p^{\pm}}$ with  generating function $S:\mathcal D\to\RR$. The main results in this section are Proposition \ref{proposition:equivalence_invariance_minimizing} and Proposition \ref{proposition:compactness}.

We first recall some basic notions related to the orbits of $F$, and we refer the reader to \cite{Gole} for more details about them. The projection 
$\pi_q: \RR^{2}\to\RR$ onto the $q$-component induces a bijection between \textit{orbits} $(q_k,p_k)_{k\in\ZZ}$ of $F$, where for all $k\in\ZZ$
$$F(q_k,p_k)=(q_{k+1},p_{k+1}),$$
and so-called \textit{stationary configurations} $(q_k)_{k\in\ZZ}$, which are sequences satisfying for all $k\in\ZZ$
$$(q_k,q_{k+1})\in\mathcal D
\qquad\text{and}\qquad
\partial_2S(q_{k-1},q_k)+\partial_1S(q_k,q_{k+1})=0.$$

An orbit $(q_k,p_k)_{k\in\ZZ}$ or its associated stationary configuration $(q_k)_{k\in\ZZ}$ are called \textit{minimal} if for any integers $u\leq v$ the family $(q_k)_{u\leq k\leq v}$ minimizes the functionnal 
$$\underline x=(x_k)_{u\leq k\leq v}\longmapsto\sum_{k=u}^{v-1}S(x_k,x_{k+1})$$
among all families $\underline x\in\RR^{v-u+1}$ with $x_u=q_u$ and $x_v=q_v$.

It is known (see  \cite[Proposition 6]{Arnaud} or \cite[Proof of Theorem A]{Bialy} or), that a minimal orbit $(q_k,p_k)_{k\in\ZZ}$ has no \textit{conjugate points}, which means that any two points $(q_k,p_k)$ and $(q_{\ell},p_{\ell})=F^{\ell-k}(q_k,p_k)$ along the orbit satisfy
$$\partial_p \left( \pi_q \circ F^{\ell-k}\right)(q_k,p_k)\neq 0.$$

\medskip

\begin{proposition}
\label{proposition:equivalence_invariance_minimizing}
Let $F: \mathbb A_{p^{\pm}} \longrightarrow \mathbb A_{p^{\pm}}$ be an exact-symplectic twist map  and $\Gamma\subset\mathbb A_{p^{\pm}}$ be an $(m,n)$-periodic Lipschitz continuous graph of $F$, for some $m\in \ZZ, n\in \NN^*$ coprime. Then:
\begin{itemize}
\item[{\bf (i)}] $\Gamma$ is invariant by $F$;
\item[{\bf (ii)}] the projection of an orbit intersecting $\Gamma$ is a minimal configuration, and hence has no conjugate points;
\item[{\bf (iii)}] $\Gamma$ is as smooth as $F$ is;
\item[{\bf (iv)}]  $F$ has no other $(m,n)$-periodic Lipschitz continuous graph.\\
\end{itemize}
\end{proposition}

\begin{remark}
Note that whereas there can exist at most one invariant rotational curve for each irrational rotation number in the twist interval (see for instance \cite[Theorem 13.2.9]{KatokHasselblatt}), this is not the case for rational rotation number. In fact, two invariant graphs with the same rational rotation number might coexist: in this case, their intersection consists of periodic orbits and the two graphs should contain also non-periodic ones. For instance, consider the twist map corresponding to an elliptic billiard (non-circular):  its phase space contains two invariant graphs of rotation number $1/2$.
\end{remark}

\begin{proof}
It is well-known that Items (i) and (ii) are equivalent, see for example \cite[Theorem 17.4]{ForniMather} for the implication (i) $\Rightarrow$ (ii), and \cite[Proposition 2.5]{AMS}  for the reverse implication. So to prove the result it is enough to show that items (i), (iii) and (iv) are satisfied.


\textbf{(i)} To show that $\Gamma$ is invariant by $F$, we {extend $F$  to a symplectic twist map $G:\RR^2\to\RR^2$}, coinciding with $F$ in a neighborhood of $\Gamma$  and satisfying the superlinearity condition at infinity, namely
$$\lim_{|q-Q|\to+\infty} \frac{S_G(q,Q)}{|Q-q|}=+\infty,$$
{where $S_G$ denotes the generating function of $G$}.
If such a $G$ exists, then the orbits of $G$ intersecting $\Gamma$ are minimal by \cite[Proposition 2.5]{AMS}; therefore, $\Gamma$ is invariant by $G$, and hence by $F$.

The construction  of such a  $G$ is quite standard, and we refer to \cite[Section 8]{ForniMather}  or \cite[Lemma 8.2]{Gole}. It is however not necessarily analytic, but this does not affect the result. {Let us assume that $F$ is positive}, and let $S$ be the generating function of $F$: it is defined on an open set $\mathcal D$. Consider the set of pairs $(q,Q)\in \mathcal D$ such that there is a point $x\in\Gamma$ and an integer $k\in \{0,\dots,n-1\}$ for which $\pi_q\circ F^k(x)=q$ and $\pi_q\circ F^{k+1}(x)=Q$; this set is contained in a compact set $K\subset \mathcal D$. On this compact set the twist condition is uniform, which means that there is {$a>0$} such that $\partial_{12}^ 2S_{|K}<-a$. Hence applying  
\cite[Section 8]{ForniMather}, or \cite[Lemma 8.2]{Gole} we deduce the existence of a symplectic twist map $G:\RR^2 \longrightarrow \RR^2$ whose generating function $S_G$ is defined on $\RR^2$, coincides with $S$ on $K$ and has the uniform twist everywhere, {that is $\partial_{12}^2S_G < -a'$ everywhere, for some $a'>0$.  This implies together with \cite[Proposition 11.2]{Gole} that the map $G$ has the announced properties, which concludes the result.}

\textbf{(iii)} The smoothness comes from the property that an orbit corresponding to an action-minimizing configuration has no conjugate points (see the proof of \cite[Theorem A]{Bialy} or \cite[Proposition 6]{Arnaud}). Hence following the argument of \cite[Proposition 2.5]{AMS}, this implies that the map $R:(q,p)\mapsto \pi_q\circ F^n(q,p)-q-m$ which vanishes on $\Gamma$ satisfies $\partial_p R(q,p)\neq 0$ for $(q,p)\in\Gamma$. By the implicit mapping theorem, $\Gamma$ can be described locally as a graph of a map which is as smooth as $F$, which proves the assertion.

\textbf{(iv)} We apply \cite[Lemma 13.2.10]{KH_book}: if $F$ has an invariant curve $\Gamma$ of rotation number $m/n$, then any order-preserving orbit of $F$ whose closure is distinct from $\Gamma$ has  rotation number $\neq m/n$ (the definition of order-preserving orbit is given in \cite{KH_book}, and periodic orbits are a particular case of them). Hence any periodic orbit of rotation number $m/n$ should have a point on  $\Gamma$, and consequently they must be entirely contained in $\Gamma$ since $\Gamma$ is invariant by $F$.
\end{proof}

Let us show a localisation result when one consider a family of symplectic twist maps.

\begin{proposition}
\label{proposition:compactness}
Assume that $I\subset\RR$ is a {compact interval}  and $(m,n)\in\ZZ\times\NN^{\ast}$. Suppose that for any $\varepsilon\in I$ we are given an exact symplectic twist map $F_{\varepsilon}$ such that:
\begin{itemize}
\item[{\bf (i)}] the map $(\varepsilon,q,p)\in \mathbb A_{I,p^{\pm}}\longmapsto F_{\varepsilon}(q,p)$ is $\class^1$-smooth;
\item[{\bf (ii)}] $m/n \in TI(F_{\varepsilon})$  for any $\varepsilon\in I$.
\end{itemize}
Then, there is a neighborhood $U$ of $\graph(p^-)\cup\graph(p^+)$ in $I\times\RR\times \RR$ such that for any $\varepsilon\in I$, any $(m,n)$-periodic orbit of $F_{\varepsilon}$ is contained in $K:=\mathbb A_{I,p^{\pm}}\setminus U$.
\end{proposition}

From this result together with Remark \ref{remark:Birkhoff_theorem}, we immediately deduce the following result.

\begin{corollary}
\label{corollary:equiboundedness}
Under the assumption of Proposition \ref{proposition:compactness}, there is a constant $k>0$ {depending only on $\inf_{K}\partial_p (\pi_q \circ F)$} such that for any $\varepsilon \in I$, an $(m,n)$-periodic invariant graph $\Gamma$ of $F_{\varepsilon}$ is Lipschitz continuous with Lipschitz constant $k$.
\end{corollary}


In order to prove Proposition \ref{proposition:compactness}, we need the following Lemma, whose proof can be found in \cite[Theorem 9.3.7 and its proof]{KatokHasselblatt}. 

\begin{lemma}
\label{lemma:safe_space_boundary}
Let $F$ be an exact-symplectic twist map of the open annulus $\mathbb A_{p^{\pm}}$ and $m/n\in TI(F)$. Then, given $a,b\in TI(F)$ such that $a<m/n<b$, there exists a neighborhood $U_-$ of $\graph(p^-)$ and a neighborhood $U_+$ of $\graph(p^+)$ such that for any $k\in\{0,\ldots,n-1\}$
\begin{equation}
\label{equation:bounded_configuration_1}
\forall(q,p)\in U_-
\qquad \pi_q\circ F^{k+1}(q,p)-\pi_q\circ F^{k}(q,p)<a
\end{equation}
and
\begin{equation}
\label{equation:bounded_configuration_2}
\forall(q,p)\in U_+
\qquad \pi_q\circ F^{k+1}(q,p)-\pi_q\circ F^{k}(q,p)>b.
\end{equation}
\end{lemma}

We can now prove Proposition \ref{proposition:compactness}.\\

\begin{proof}[Proof of Proposition \ref{proposition:compactness}]
Consider a decreasing sequence $U_j$ of neighborhoods of $\graph(p^-)\cup\graph(p^+)$ in $I\times\RR\times \RR$ such that $\cap_j U_j = \graph(p^-)\cup\graph(p^+)$ and assume by contradiction that there are two sequences $(\underline{x}^{(j)})_j$ and $(\varepsilon_j)_j$ such that for each $j$, $\underline{x}^{(j)}$ is an $(m,n)$-periodic orbits of $F_{\varepsilon_j}$ having a point in $(\{\varepsilon_j\}\times\mathbb A_{I,p^{\pm}}^{\varepsilon_j})\cap U_j$. 

Since $I$ is compact, we can assume that the sequence $(\varepsilon_j)_j$ converges to a certain $\varepsilon\in I$. Consider the sets $U_-$ and $U_+$ of Lemma \ref{lemma:safe_space_boundary} associated to the rotation number $m/n$ and the twist map $F_{\varepsilon}$. Define $U=U_-\cup U_+$. By continuity of $F$, $p^-$ and $p^-$, for $\varepsilon'$ sufficiently close to $\varepsilon$, $U$ remains a neighborhood of $\graph(p^-)(\varepsilon',\cdot)\cup\graph(p^+)(\varepsilon',\cdot)$ and Equations \eqref{equation:bounded_configuration_1} and \eqref{equation:bounded_configuration_2} are also satisfied by $F_{\varepsilon'}$. In particular, considering $\varepsilon'=\varepsilon_j$ for sufficiently large $j$, $(m,n)$-periodic orbits of $F_{\varepsilon_j}$ have their points in $\mathbb A_{I,p^{\pm}}^{\varepsilon_j}\setminus U$. {Now by the assumptions made on the sequence of $U_j$'s, if $j$ is sufficiently large $U_j\subset I\times U$, and the latter implies that $\underline{x}^{(j)}$ is an $(m,n)$-periodic orbit of $F_{\varepsilon_j}$ having a point in $U$.} This is contradictory and concludes the proof.
\end{proof}

\section{Proof of Main Theorem 1 (Theorem \ref{theorem:main2})}
\label{section:proof_main}

In this section, we prove  Main Theorem 1, namely Theorem \ref{theorem:main2}. Given a pair $(m,n) \in \ZZ\times \NN^*$ of coprime integers, we first recall some properties of the set 
$$\mathcal{I}_{(m,n)}(\RR) = \{\varepsilon\in I\,|\, F_{\varepsilon}\text{ has an }(m,n)\text{-periodic invariant graph}\}$$
defined in the statement of the theorem.


\begin{lemma}
\label{lemma:extension_lemma27}
Under the assumptions of Theorem \ref{theorem:main2}.
\begin{itemize}
\item[{\bf (i)}] For $\varepsilon\in \mathcal{I}_{(m,n)}(\RR)$, there exists a unique {$1$-periodic} continuous map $\gamma_{\varepsilon}:\RR\to\RR$ such that $\text{graph}(\gamma_{\varepsilon})$ is an $(m,n)$-periodic  graph invariant by $F_\varepsilon$.
\item[{\bf (ii)}] The map $(\varepsilon,q)\in \mathcal{I}_{(m,n)}(\RR)\times\RR\longmapsto\gamma_{\varepsilon}(q)$ is continuous (for the topology induced by $I\times\RR$ on $\mathcal{I}_{(m,n)}(\RR)\times\RR$).
\item[{\bf (iii)}] $\mathcal{I}_{(m,n)}(\RR)$ is a closed subset of $I$.
\end{itemize}
\end{lemma}

\begin{proof} 

\textbf{(i)} The existence of such $\gamma_\varepsilon$ follows from the definition of $\mathcal{I}_{(m,n)}(\RR)$. The unicity comes from Proposition \ref{proposition:equivalence_invariance_minimizing}. {Note that $\gamma_{\varepsilon}$ is Lipschitz continuous with a Lipschitz constant depending on $\varepsilon$, but which can be chosen uniformly for any $\varepsilon$ lying in a compact subinterval $I'\subset I$ by Corollary \ref{corollary:equiboundedness}. This property will be useful for the rest of the proof.}

\textbf{(ii)} Choose $\varepsilon\in\mathcal{I}_{(m,n)}(\RR)$ and a sequence $(\varepsilon_j)_j\subset\mathcal{I}_{(m,n)}(\RR)$ converging to $\varepsilon$. Let us show that $(\gamma_{\varepsilon_j})_j$ converges to $\gamma_{\varepsilon}$ in the uniform topology. {As noticed in item (i), each $\gamma_{\varepsilon_j}$ as well as $\gamma_{\varepsilon}$ is Lipschitz continuous, and they share the same Lipschitz constant ({\it i.e.}, they are equi-Lipschitz).}

We do the proof in two steps:

\textbf{$\bullet$ Step 1.} We prove that any subsequence of $(\gamma_{\varepsilon_j})_j$ converging in the space of $1$-periodic continuous maps $\RR\to\RR$ {for the uniform topology} has $\gamma_{\varepsilon}$ as a limit.

\textbf{$\bullet$ Step 2.} We show that $(\gamma_{\varepsilon_j})_j$ has at least one converging subsequence (applying Ascoli-Arzelà theorem).

{These two steps imply the assertion of item (ii)}.\\

\textit{Proof of Step 1.} Assume that there is a subsequence $(\gamma_{\varepsilon_{j_k}})_k$ converging to a {$1$-periodic} map $\gamma_{\infty}:\RR\to\RR$ {in the uniform topology}. By Proposition \ref{proposition:compactness}, $\graph({\gamma_{\infty}})\subset \mathbb A_{I,p^{\pm}}^{\varepsilon}$. The identity $F^n_{\varepsilon_{j_k}}=\text{Id}+(m,0)$ is satisfied on $\graph\gamma_{\varepsilon_{j_k}}$, hence, by continuity,  $F^n_{\varepsilon}=\text{Id}+(m,0)$ on $\graph ({\gamma_{\infty}})$. Let us mention that $\gamma_{\infty}$ is Lipschitz continuous since all $\gamma_{\varepsilon_{j_k}}$ are Lipschitz continuous with the same Lipschitz constant. Hence $\graph \gamma_{\infty}$ is a $(m,n)$-periodic graph of $F_{\varepsilon}$, which is invariant by Proposition \ref{proposition:equivalence_invariance_minimizing}. Therefore $\gamma_{\infty} = \gamma_{\varepsilon}$, which follows from the unicity of $(m,n)$-periodic graphs, again by Proposition \ref{proposition:equivalence_invariance_minimizing}.
%

\textit{Proof of Step 2.} Let us check that the assumptions of Ascoli-Arzelà theorem are satisfied by the sequence of maps $(\gamma_{\varepsilon_j})_j$. By Proposition \ref{proposition:compactness}, the maps $\gamma_{\varepsilon_j}$ are contained in a compact subset $K$ of $\mathbb A_{I,p^{\pm}}$ which implies that they are bounded by the same constant ({\it i.e.}, equibounded). Moreover, we have already noticed that they are Lipshitz continuous with a uniform Lipschitz constant. Hence Step 2 follows from Ascoli-Arzelà theorem, and therefore point (ii) is proven.
 
\textbf{(iii)} Consider a sequence of $\varepsilon_j\in\mathcal{I}_{(m,n)}(\RR)$  converging to a $\varepsilon\in I$. {As in item (ii), the family of maps
$\gamma_{\varepsilon_j}$ is equibounded and equi-Lipschitz.} By extracting a subsequence as in Step 2 of item (ii), we can suppose that $\gamma_{\varepsilon_j}$ converges to a $1$-periodic Lipschitz continuous map ${\gamma}_{\varepsilon}:\RR\to\RR$, such that $\{\varepsilon\}\times\graph(\gamma_{\varepsilon})\subset \mathbb A_{I,p^{\pm}}$ (Proposition \ref{proposition:compactness}), which is $(m,n)$-periodic by continuity of $F$ as a map of $(\varepsilon,q,p)$. By Proposition \ref{proposition:equivalence_invariance_minimizing}, this implies that $\varepsilon\in\mathcal{I}_{(m,n)}(\RR)$. \\
\end{proof}

\medskip

We can now prove the following stronger result on the topological structure of the set $\mathcal{I}_{(m,n)}(\RR)$.\\

\begin{lemma}
\label{lemma:non_empty_interior}
Under the assumptions of Theorem \ref{theorem:main2}, the set $\mathcal{I}_{(m,n)}(\RR)$ is either the whole $I$ or it has empty interior.
\end{lemma}

\begin{proof}
Assume that $\mathcal{I}_{(m,n)}(\RR)$ has non-empty interior, and let us show that $\mathcal{I}_{(m,n)}(\RR)=I$. Let us define $a:=\inf I$ and $b:=\sup I$, hence  $I\cap(a,b) = (a,b)$. 

Consider a connected component $A\subset\mathcal{I}_{(m,n)}(\RR)$ which is not reduced to a point. Let $\beta=\sup A$, and suppose that $\beta<b$. We will raise a contradiction. 

First note that necessarily $\beta\in\mathcal{I}_{(m,n)}(\RR)$ since $\beta\in I$ and the set $\mathcal{I}_{(m,n)}(\RR)$ is closed in $I$ by Lemma \ref{lemma:extension_lemma27}; therefore, $\beta\in A$.

Applying again Lemma \ref{lemma:extension_lemma27}, we can then define a family of Lipschitz continuous maps $(\gamma_{\varepsilon})_{\varepsilon\in A}$ such that for all $\varepsilon\in A$, $\text{graph}(\gamma_{\varepsilon})$ is an $(m,n)$-periodic  graph invariant by $F_{\varepsilon}$. Moreover the map {$\Gamma:(\varepsilon,q)\mapsto \gamma_{\varepsilon}(q)$} is continuous, again by Lemma \ref{lemma:extension_lemma27}. We will show that we can extend {$\Gamma$  to the open set $(\beta-r,\beta+r) \times \RR$}, with $r>0$,  thus leading to a contradiction, {since this would imply that $(\beta-r, \beta+r) \subseteq A$, thus contradicting the maximality of $A$}.


Let us apply the implicit function theorem to the map 
$$\Delta_1(\varepsilon,q,p):=\pi_q\circ F^n_{\varepsilon}(q,p)-q-m.$$
Since $F_{\beta}$ has no conjugate points on $\text{graph}(\gamma_{\beta})$ (see Proposition \ref{proposition:equivalence_invariance_minimizing}), $\partial_p\Delta_1$ do not vanish on the set $\{(\beta,q,\gamma_{\beta}(q))\,|\, q\in\RR\}$.
Hence we can define an analytic map $(\varepsilon,q)\in(\beta-r,\beta+r)\times\RR\mapsto\eta_{\varepsilon}(q)$, with $r>0$, such that $\eta_{\beta}=\gamma_{\beta}$ and for any $(\varepsilon,q,p)$ close to $(\beta,q,\eta_{\beta}(q))$ we have 
$$\pi_q\circ F_{\varepsilon}(q,p)=q+m\qquad\Leftrightarrow\qquad p=\eta_{\varepsilon}(q).$$

The latter implies -- together with the continuity of $\gamma$ -- that for $\varepsilon<\beta$ and sufficiently close to $\beta$, we have $\eta_{\varepsilon}=\gamma_{\varepsilon}$. 

Now consider, the map $\Delta_2:(\beta-r,\beta+r)\times\RR\to\RR$ defined by
$$(\varepsilon,q)\mapsto \pi_p\circ F^n_{\varepsilon}(q,\eta_{\varepsilon}(q))-\eta_{\varepsilon}(q).$$
Since $\eta_{\varepsilon}=\gamma_{\varepsilon}$ for $\beta-r<\varepsilon\leq\beta$ and $\graph \gamma_{\varepsilon}$ is $(m,n)$-periodic, $\Delta_2$ vanishes on $(\beta-r,\beta]\times\RR$. Since $\Delta_2$ is analytic and $(\beta-r,\beta+r)\times\RR$ is connected, $\Delta_2$ vanishes on $(\beta-r,\beta+r)\times\RR$. 

Combining the two results on $\Delta_1$ and $\Delta_2$, we obtain that, for any $\varepsilon\in(\beta-r,\beta+r)$, $\text{graph}(\eta_{\varepsilon})$ is an $(m,n)$-periodic graph of $F_{\varepsilon}$. This contradicts the fact that $\varepsilon$ cannot be bigger than $\beta$.
\end{proof}

\begin{proof}[Proof of Theorem \ref{theorem:main2}]
Let us suppose that $\mathcal{I}_{(m,n)}(\RR)$ has an accumulation point $\beta\in I$ and show that in this case $\mathcal{I}_{(m,n)}(\RR)=I$. The proof of Lemma \ref{lemma:non_empty_interior} can be adapted to this context: there is a sequence $(\varepsilon_n)_n$ converging to $\beta$ with $\varepsilon_n\neq\beta$ satisfying 
$$\forall(n,q)\in\NN\times\RR\qquad\qquad\Delta_2(\varepsilon_n,q)=0.$$
Hence $\Delta_2$ is flat at any $(\beta,q)\in\{\beta\}\times\RR$, meaning that its partial derivatives of any order in $\varepsilon$ and $q$ vanish. By analyticity of $\Delta_2$ and $1$-periodicity in $q$, it vanishes on an open set of the form $J\times\RR$. Hence for $\varepsilon$ sufficiently close to $\beta$, $\eta_{\varepsilon}$ parametrizes an $(m,n)$-periodic $F_{\varepsilon}$-invariant graph. We conclude that $\mathcal{I}_{(m,n)}(\RR)$ has non-empty interior, and by Lemma \ref{lemma:non_empty_interior} that $\mathcal{I}_{(m,n)}(\RR)=I$.
\end{proof}

\section{Proof of Theorem \ref{theorem:birkhoff_billiard}}
\label{section:proof_inner_billiard}

Let $(\Omega_{\varepsilon})_{\varepsilon\in I}$ be an analytic family of strictly convex domains with analytic boundary. By applying homotheties to the different $\Omega_{\varepsilon}$, we can suppose that each $\Omega_{\varepsilon}$ has perimeter $1$ (note that homotheties do not break the property of the billiard map in a domain of having an $(m,n)$-periodic invariant graph of a fixed rotation number). 

For each $\varepsilon\in I$, one can consider a parametrization $s\in\RR\mapsto\gamma_{\varepsilon}(s)$ of $\gamma_{\varepsilon}$ by arc-length which by assumption one can assume analytic in $(\varepsilon,s)$. The billiard map inside $\Omega_{\varepsilon}$ induces an exact symplectic twist map $F_{\varepsilon}:\RR\times(-1,1)\to\RR\times(-1,1)$ defined for all $(s,\sigma)\in \RR\times(-1,1)$ by
$$F_{\varepsilon}(s,\sigma)=(s_1,\sigma_1)$$
where, if $\sigma=-\cos\varphi$ and $\varphi\in(0,\pi)$, then $s_1$ and $\sigma_1$ are defined by following requirements:
\begin{itemize}
\item[({\it a})] $\gamma_{\varepsilon}(s_1)$ is the second point of intersection with $\partial\Omega_{\varepsilon}$ of the line $\ell$ passing through $\gamma_{\varepsilon}(s)$ and making an angle $\varphi$ with the tangent line of $\partial\Omega$ at $\gamma_{\varepsilon}(s)$;
\item[({\it b})] the line $\ell$ makes an angle $\varphi_1$ with the tangent line of $\partial\Omega$ at $\gamma_{\varepsilon}(s_1)$ and this defines $\sigma_1=-\cos\varphi_1$.
\end{itemize}

Let us check that the corresponding family of billiard maps $(F_{\varepsilon})_{\varepsilon \in I}$ satisfies the assumptions of Theorem \ref{theorem:main2}. 

It is well-known that the billiard map written in the coordinates $(s,\sigma)$ is an exact symplectic twist map as defined in Definition \ref{definition:general_symplectic_twist_map}.

Given $\varepsilon\in I$ and $(s,\varphi)\in\RR\times(-1,1)$, the corresponding line $\ell$ (as in ({\it a})) is transverse to the boundary of $\Omega_{\varepsilon}$ at $\gamma_{\varepsilon}(s_1)$. Applying the implicit function theorem, one deduces that there exist a neighborhood $U$ of $(\varepsilon,s,\varphi)\in I\times \RR\times(-1,1)$ and an analytic map $\varphi:U\to\RR$ such that $s_1(\varepsilon',s',\sigma')=\varphi(\varepsilon',s',\sigma')$ for $(\varepsilon',s',\sigma')\in U$. The analytic regularity of $\sigma_1$ comes from the fact that it can be written as 
$$\sigma_1 = \frac{\gamma_{\varepsilon}(s_1)-\gamma_{\varepsilon}(s)}{\|\gamma_{\varepsilon}(s_1)-\gamma_{\varepsilon}(s)\|}\cdot \gamma_{\varepsilon}'(s_1)$$
where $u\cot v$ denotes the scalar product of two vectors $u$ and $v$. Hence Assumption (i) of Theorem \ref{theorem:main2} is satisfied.

It is well-known that since each $\Omega_{\varepsilon}$ is strictly convex, the map $F_{\varepsilon}$ extends to a continuous map $\RR\times[-1,1]$ satisfying for any $s\in\RR$ the equalities $F_{\varepsilon}(s,0)= (s,0)$ and $F_{\varepsilon}(s,1)=(s+1,1)$. This implies that $\text{TI}(F_{\varepsilon})=(0,1)$.

Hence the assumptions of Theorem \ref{theorem:main2} are satisfied and it implies the result.

\section{Proof of Theorem \ref{theorem:outer_billiard}}
\label{section:proof_outer_billiard}

Let $(\Omega_{\varepsilon})_{\varepsilon\in I}$ be an analytic family of strictly convex domains with analytic boundary. We introduce the so-called \textit{envelope coordinates} on each $\partial\Omega_{\varepsilon}$, see \cite{Boyland}. They are defined as follows. 

By applying translations to the domains, one can assume that there is a point $O$ which remains inside all domains. Consider a fixed direction $Ox$. For each $\varepsilon\in I$ and any angle $\theta\in\RR$, one can associate the oriented line $\mathcal L_{\theta}$ to $\partial\Omega$, which makes an angle $\theta+\pi/2$ with $Ox$ and is tangent to it at a point $\alpha(\theta)$ where the orientations of $\partial\Omega$ and  $\mathcal L_{\theta}$ are the same. Let $p_{\varepsilon}(\theta)$ be the distance from $O$ to $\mathcal L_{\theta}$. Under the assumptions of Theorem \ref{theorem:outer_billiard}, one can assume that $p$ is analytic in $(\varepsilon,\theta)$.

For any point $p\in\RR^2\setminus\Omega$, one can find a unique $\theta$ such that the vector $\alpha(\theta)-z$ and $\mathcal{L}_{\theta}$ are collinear and with the same orientation. The pair $(\theta,\gamma)$ where $\gamma=\|\alpha(\theta)-z\|^2/2$ is called the envelope coordinate of $p$ and uniquely determines $p$.

The outer billiard map outside $\Omega_{\varepsilon}$ acts on the space of enveloppe coordinates $\RR\times(0,+\infty)$ and induces an exact symplectic twist map $F_{\varepsilon}:\RR\times(0,+\infty)\to\RR\times(0,+\infty)$ defined for all $(\theta,\gamma)\in \RR\times(0,+\infty)$ by
$$F_{\varepsilon}(\theta,\gamma)=(\theta_1,\gamma_1).$$

Let us check that the corresponding family of outer billiard maps $(F_{\varepsilon})_{\varepsilon \in I}$ satisfies the assumptions of Theorem \ref{theorem:main2}. 

Given $\varepsilon\in I$ and $p\in\RR^2\setminus\overline\Omega_{\varepsilon}$, consider $G_{\varepsilon}(p)$ to be the image of $p$ after one outer billiard reflection on $\Omega_{\varepsilon}$. The map $G$ is well-defined and analytic. It is then a consequence of the implicit function theorem that the enveloppe coordinates of a point $q$ depend analytically on $q$. Hence Assumption (i) of Theorem \ref{theorem:main2} is satisfied.

It is well-known that the twist interval of a dual billiard map is $\text{TI}(F_{\varepsilon})=(0,1/2)$. Hence the assumptions of Theorem \ref{theorem:main2} are satisfied and it implies the result.

{
\section{Proof of Theorem \ref{theorem:symplectic_billiard}}
\label{section:proof_symplectic_billiard}

Let us first recall some basic properties of symplectic billiards; we refer to  \cite{AlbersTabachnikov} for more details. Consider a strictly convex bounded domain $\Omega$ with smooth boundary and parametrized by a periodic map $\gamma:\RR\to\RR^2$; without loss of generality, let us assume here that the period is $1$. Given $t\in[0,1)$, there is a unique $t^{\ast}\in[0,1)$ different from $t$ and such that the tangent lines to $\partial\Omega$ at $\gamma(t)$ and at $\gamma(t^{\ast})$ are parallel.

The symplectic billiard map is given  by a map 
\begin{eqnarray*}
F:  \mathcal P&\to&\mathcal P\\
(t_1,t_2)&\mapsto&(t_2,t_3)
\end{eqnarray*}
where $F(t_1,t_2)=(t_2,t_3)$ if and only if $\gamma(t_1)$, $\gamma(t_2)$, and $\gamma(t_3)$ are successive reflection points and $\mathcal P$ is the set
$$\mathcal P = \{(t_1,t_2)\in\RR^2\,|\,t_1<t_2<t_1^{\ast}\}.$$
Let $\omega(x,y) = dx\wedge dy$ the canonical symplectic form on $\RR^2_{(x,y)}$. The map $S:\mathcal P\to\RR$ defined by 
$$S(t_1,t_2) = -\omega(\gamma(t_1),\gamma(t_2))$$
is a generating map for $S$ in the following sense: for any $(t_1,t_2)\in\mathcal P$ and $(t_2,t_3)\in\mathcal P$, 
\begin{equation}
\label{equation:generating_symplectic}
F(t_1,t_2)=(t_2,t_3)\quad\Longleftrightarrow\quad \partial_2S(t_1,t_2)+\partial_1S(t_2,t_3)=0.
\end{equation}

Let us relate the map $F$, the space $\mathcal P$ and the map $S$ to our  Definition \ref{definition:general_symplectic_twist_map}. Observe that the map
$\varphi$ defined by $\varphi(t_1,t_2) = (t_1,s_1)$ where 
$$s_1 = -\partial_1S(t_1,t_2) = \omega(\gamma'(t_1),\gamma(t_2))$$
is a diffeomorphism from $\mathcal P$ on its image $\varphi(\mathcal P)$, which is simply 
$\mathbb{A}_{s^{\pm}}$, with 
$$s^-(t) = \omega(\gamma'(t),\gamma(t))
\qquad\text{and}\qquad
s^+(t) = \omega(\gamma'(t),\gamma(t^{\ast})).$$
Therefore $F$ can be rewritten in coordinates $(t,s)$ as a map of the space $\mathbb{A}_{s^{\pm}}$. If $F(t_1,s_1)=(t_2,s_2)$, then 
$$\frac{\partial t_2}{\partial s_1} = \frac{1}{\omega(\gamma'(t_1),\gamma'(t_2))} >0,$$
hence the twist condition is satisfied. Finally $S$ is a generating map in the sense of Definition \ref{definition:general_symplectic_twist_map} since 
$$s_2dt_2 - s_1dt_1 = d(\omega(\gamma(t_1),\gamma(t_2))) = dS,$$
as it follows from \eqref{equation:generating_symplectic}. We deduce the following:

\begin{proposition}
\label{proposition:setting_symplectic_billiard_map}
The symplectic billiard dynamics can be modelled by the exact-symplectic twist-map $F:\mathbb{A}_{s^{\pm}}\to\mathbb{A}_{s^{\pm}}$.
\end{proposition}

Let $(\Omega_{\varepsilon})_{\varepsilon\in I}$ be an analytic family of strictly convex domains with analytic boundary. By applying homotheties to the different $\Omega_{\varepsilon}$, we can suppose that each $\Omega_{\varepsilon}$ has perimeter $1$ (note that homotheties do not break the property of the symplectic billiard map in a domain of having an $(m,n)$-periodic invariant graph of a fixed rotation number). 

Let us denote by $(\varepsilon,t)\in I\times\RR\mapsto \gamma_{\varepsilon}(t)$ an analytic map which is $1$-periodic in $t$ and such that for any $\varepsilon$, $\gamma_{\varepsilon}$ parametrizes $\partial\Omega_{\varepsilon}$. Following Proposition \ref{proposition:setting_symplectic_billiard_map}, this induces a family of maps $s_{\varepsilon}^{\pm}:\RR\to\RR$ and of exact-symplectic twist maps $F_{\varepsilon}:\mathbb{A}_{s^{\pm}_{\varepsilon}}\to\mathbb{A}_{s^{\pm}_{\varepsilon}}$.

It can be checked using the implicit function theorem that the family $(F_{\varepsilon})_{\varepsilon \in I}$ is analytic and defined on $\mathbb{A}_{I,s^{\pm}}$. Moreover the twist interval of the symplectic billiard map is $(0,1)$ by the same continuity argument as for the classical billiard (see Section \ref{section:proof_inner_billiard}).

Hence the assumptions of Theorem \ref{theorem:main2} are satisfied and this implies the result.}

\section{Proofs of Main Theorem 2 (Theorems \ref{theorem:main_Banach}) and Theorem \ref{theorem:main_Banach_billiard}}
\label{section:proof_main_Banach}

Let us recall that $\mathscr T^{\omega}$ is the set of analytic twist-maps of $\mathbb A_{p^{\pm}}$ which admit a continuous extension to the closure $\overline{\mathbb A_{p^{\pm}}}$; the set $\mathscr T_{m,n}^{\omega}$ consists of analytic twist-maps of $\mathscr T^{\omega}$ having an $(m,n)$-periodic invariant graph. 
For any $F\in \mathscr T_{m,n}^{\omega}$ there is an analytic $1$-periodic map $\gamma_{F}:\RR\to\RR$ such that $\text{graph}(\gamma_{F})$ is the unique $(m,n)$-periodic invariant graph of the map $F$.

\begin{lemma}
\label{lemma:Banach_closedness}
The set $\mathscr T_{m,n}^{\omega}$ is closed in $\mathscr T^{\omega}$.
\end{lemma}

\begin{proof}
Let $(F_k)_k$ be a sequence of maps in $\mathscr T_{m,n}^{\omega}$ converging to a map $F\in \mathscr T^{\omega}$. As a consequence of Lemma \ref{lemma:safe_space_boundary}, one can find an open neighborhood $U$ of $\graph(p^-)\cup\graph(p^+)$ in $\mathbb A_{p^{\pm}}$ such that for any map $G$ close to $F$ and for any pair $(q,p)\in\mathbb A_{p^{\pm}}$ satisfying $G^n_{\Omega}(q,p)=(q+m,p)$ has the property
$$(q,p)\in \mathbb A_{p^{\pm}}\setminus U.$$
This fact together with Remark \ref{remark:Birkhoff_theorem} imply that for $k$ large enough the graph of the map $\gamma_{F_k}$ lies in the set $\mathbb A_{p^{\pm}}\setminus U$;  hence that the maps $\gamma_{F_k}$ are Lipschitz-continous with the same Lipshitz constant. Ascoli-Arzelà theorem applies: there is a subsequence $(\gamma_{F_{n_k}})_k$ converging to a continuous $1$-periodic map $\gamma_{\infty}:\RR\to \RR$ whose graph  also lies in $\mathbb A_{p^{\pm}}\setminus U$ and satisfies the following identity
$$\forall q\in\RR\qquad F^n(q,\gamma_{\infty}(q))=(q+m,\gamma_{\infty}(q)).$$
Proposition \ref{proposition:equivalence_invariance_minimizing} implies then that $F$ has an $(m,n)$-periodic invariant graph, hence $F\in \mathscr T_{m,n}^{\omega}$.
\end{proof}

\begin{lemma}
\label{lemma:Banach_continuity_graphs}
The map $(F,q)\in \mathscr T_{m,n}^{\omega}\times \RR\mapsto\gamma_{F}(q)\in\RR$ is continuous for the induced topology.
\end{lemma}

\begin{proof}
Let $(F_k)_k$ be a sequence of maps in $\mathscr T_{m,n}^{\omega}$ converging to a map $F\in \mathscr T^{\omega}_{m,n}$. Using the same notations and arguments as in the proof of Lemma \ref{lemma:Banach_closedness}, for any $k$ the map $\gamma_{F_k}$ is Lipschitz-continuous and we can choose the Lipschitz-constant uniformly in $k$. Hence it has at least one converging subsequence for the uniform topology. Now any converging subsequence $\gamma_{\infty}$ parametrizes an $(m,n)$-periodic invariant graph of $F$ by continuity. Hence by Proposition \ref{proposition:equivalence_invariance_minimizing}, $\gamma_{\infty}=\gamma_{F}$ and all converging subsequences of $(\gamma_{F_k})_k$ have the same limit. This implies that the sequence $(\gamma_{F_k})_k$ converges to $\gamma_{F}$, which concludes the proof.
\end{proof}

\begin{proof}[Proof of  Main Theorem 2 (Theorem \ref{theorem:main_Banach})]
As stated by Lemma \ref{lemma:Banach_closedness}, the set $\mathscr T_{m,n}^{\omega}$ is closed. Let us prove that $\mathscr T_{m,n}^{\omega}$ is locally given by the zeros of an analytic map. Let $F_0\in \mathscr T_{m,n}^{\omega}$ and consider its $(m,n)$-periodic invariant graph $\gamma_{F_0}$. Since the latter has no conjugate points, by the implicit function theorem on analytic maps of Banach spaces (see \cite{Whittlesey}) there exists a neighborhood $V$ of $F_0$ in $\mathscr T^{\omega}$ and an analytic map
$$\overline p:V\times\RR\to\RR$$
such that 
\begin{enumerate}
\item for any $F\in V$, $q\mapsto\overline p_{\Omega}(q)$ is $1$-periodic;
\item $\overline p_{F_0}=\gamma_{F_0}$;
\item for any map $F$ in $V$ close to $F_0$, and any point $(q,p)$ close to $(q,\gamma_{F_0}(p))$, we have
$$\pi_1\circ F^n(q,p)=q+m
\qquad\Leftrightarrow\qquad
p=\overline p_{F}(q).$$
\end{enumerate}

By Lemma \ref{lemma:Banach_continuity_graphs}, its means that we can eventually shrink $V$ and state that $\overline p_{F}=\gamma_{F}$ for $F\in \mathscr T_{m,n}^{\omega}\cap V$. Hence we can state the equivalence
$$\forall F\in V\qquad 
F\in \mathscr T_{m,n}^{\omega}\quad
\Longleftrightarrow
\quad
\forall q\in\RR\quad
\pi_2\circ F^{n}(q,\overline p_{F}(q)) = \overline p_{F}(q).$$
Thus consider the space Banach space $\mathscr C^{}(\RR/\ZZ)$ of continuous $1$-periodic maps on $\RR$, and define the map
$$\varphi:\left\{\begin{array}{rcl}
V&\to&\mathscr C(\RR/\ZZ)\\
F&\mapsto& (q\mapsto \pi_2\circ F^{n}(q,\overline p_{F}(q))-\overline p_{F}(q))
\end{array}
\right.$$
The map $\varphi$ is analytic by construction and satisfies $\varphi^{-1}(0) = \mathscr T_{m,n}^{\omega}\cap V$. Since this is true for any $F_0\in \mathscr T_{m,n}^{\omega}$, this concludes the proof.
\end{proof}

\begin{proof}[Proof of Theorem \ref{theorem:main_Banach_billiard}]
Given a domain $\Omega\in\mathscr D_r$, we can consider the billiard map $F_{\Omega}$ inside $\Omega$. As a consequence of the implicit function theorem for analytic maps of Banach spaces (see \cite{Whittlesey}), the map $F:\mathcal D_r\to \mathscr T^{\omega}$ is an analytic. Moreover, the set of domains in $\mathscr D_r$ having an $(m,n)$-periodic invariant graph is $X_{m,n} := F^{-1}(\mathscr T^{\omega}_{m,n})$. Since $\mathscr T^{\omega}_{m,n}$ is an analytic subset of $\mathscr T^{\omega}$, the set $X_{m,n}$ is an analytic subset of $\mathscr D_r$.
\end{proof}

\appendix
\section{Remarks on billiard dynamics}

In this section we underline the peculiarity of billiard maps by proving properties on their $(m,n)$-periodic graphs, as given in Proposition \ref{proposition:equivalence_invariance_minimizing} for general twist-maps, but in a different way which requires less \textit{machinery}. It relies on the fact that the generating map $h$ of a billiard is the opposite of the distance between consecutive points, and hence if one fixes two points $x$ and $z$, the quantity
$$h(x,y)+h(y,z)$$
is minimized for $y\neq x,z$.

\begin{proposition}
\label{proposition:billiard_periodic_graphs}
Let $f:\TT^1\times(0,1)\to\TT^1\times(0,1)$ be the billiard map written in $(s,-\cos\varphi)$ coordinates and $\Gamma$ be an $(m,n)$-periodic Lipschitz graph $\Gamma$ of $f$. Then: 
\begin{itemize}
\item[{\bf (i)}] $\Gamma$ invariant by $f$;
\item[{\bf (ii)}] the projection of an orbit of $\Gamma$ is a minimal configuration;
\item[{\bf (iii)}] $\Gamma$ is as smooth as $f$ is.
\end{itemize}
\end{proposition}

\begin{proof}
Items {\bf (i)} and {\bf (iii)} follow from Proposition \ref{proposition:equivalence_invariance_minimizing}.\\
Let us show item {\bf (ii)}  First let us show that the projection $\underline s = (s_p)_{p\in\ZZ}$ of an orbit $(s_p,y_p)_{p\in \ZZ}$ and $(s_0,y_0)\in\Gamma$ is minimal (here $y$ stands for $-\cos\varphi$).

Let us write $\Gamma = \{(s,\phi(s))\,|\, s\in \RR/L\ZZ\}$ where $\phi:\TT^1\to(0,1)$ is a Lipschitz continuous map and $L$ is the perimeter of the billiard boundary. Denote by $h:D\to\RR$ the generating function of the billiard where $D=\{(s,s')\,|\, s\leq s'\leq s+1\}$. For any $p\leq q$, define
$$\mathcal E_{pq}(x_p,\ldots,x_q) = \sum_{k=p}^{q-1}h(x_k,x_{k+1})-\int_{x_p}^{x_q}\phi(u)du.$$
In the case of the billiard map, $\mathcal E_{pq}$ is well-defined and continuous on the compact set 
$K=\{(x_p,\ldots,x_q)\,|\,x_0\in[0,1],\,\forall k\,\, x_k\leq x_{k+1}\leq x_k+1\}.$ Hence it has a global minimal value which is reached at a certain point $\underline x = (x_p,\ldots,x_q)$. Now by triangular inequality, we have the important fact that
$$\underline x\in\text{int}(K).$$
This property is specific to the billiard map and due to the fact that $h(x,y)+h(y,z)<h(x,x)+h(x,z)$ for any $x<y<z$. Hence $\underline x$ is a critical point of $\mathcal E_{pq}$ and by a classical argument, $\underline x$ is a stationary configuration corresponding to an orbit $(x_j,y_j)_{p\leq j\leq q}$ of the billiard map such that $y_p=\phi(x_p)$ and $y_q=\phi(x_q)$.

It follows that our initial configuration $\underline s$ minimizes each $\mathcal E_{pn,qn}$ among all configurations with the same endpoints $s_{pn}$ and $s_{qn}$, where $p\leq q$. Hence $\underline s$ minimizes the action considered between the indices $pn$ and $qn$, among all configurations with the same endpoints. By another classical argument, $\underline{s}$ is minimal.
\end{proof}

%


\end{document}